\documentclass[12pt]{article}
\usepackage{amsmath,amssymb,amsthm,titlefoot}

\newtheorem{thm}{\noindent Theorem}[section]
\newtheorem{lem}{\noindent Lemma}[section]

\newtheorem{prop}{\noindent Proposition}[section]

\newtheorem{rem}{\noindent Remark}[section]

\def\om{\omega}

\def\Om{\Omega}

\def\beq{ \begin{align*}}
\def\eeq{ \end{align*}}

\def\beqn{\begin{equation}}
\def\eeqn{\end{equation}}

\def\1{{\bf 1}}

\def\v2{\vskip2mm}

\def\tst12{{\textstyle\frac12}}

\begin{document}
\title{The inner boundary of random walk range}
\author{Izumi Okada}
\date{}

\amssubj{Primary 60J05, Secondary 60F10}
\maketitle
\begin{abstract}
In this paper, we deal with the inner boundary of random walk range, that is, the set of those points in a random walk range which have at least one neighbor site outside the range. If $L_n$ be the number of the inner boundary points of random walk range in the $n$ steps, we prove $\lim_{n\to \infty}\frac{L_n}{n}$ exists with probability one. Also, we obtain some large deviation result for  transient walk. We find that the expectation of the number of the inner boundary points of simple random walk on two dimensional square lattice is of the same order as  $\frac{n}{(\log n)^2}$. 
\end{abstract}
\section{Introduction and Known results }
Let  $d$ be a positive integer and  $X_1,X_2,...$ be i.i.d. ${\mathbb{Z}^{d}}$-valued random variables, and put $S_k=S_0+\sum_{i=1}^{k}X_i$ with some constant  $S_0$, a random walk taking values in ${\mathbb{Z}^{d}}$ started from $S_0$. 
Let $P^a$ denote the probability law of the walk such that $S_0=a$ a.s., and 
we simply write $P$ for $P^0$.  
Let $R_n$ be the cardinality of the range of the walk of length $n$. 
Namely, $R_n$ is the number of distinct points visited 
by the walk  in the first $n$ steps.
 Many results of the asymptotic behavior of $R_n$ as $n\to \infty$ have been obtained  by various authors. 
It was shown by Spitzer \cite{SP}, pp $38-40$ that 
for all random walks of any dimension,
 \begin{align*}
\lim_{n \to \infty}\frac{R_n}{n} =v\quad a.s.
 \end{align*}
where $v= P(0\notin \{S_k\}_{k=1}^{\infty})$.  
For any random walk in $\mathbb{Z}^{d}$ with $d\ge4$ the following was shown by  
 Jain and Pruitt \cite{Jain2}: 
  \begin{align*}
&\mathrm{Var} R_n \sim c^2 n,\\
&\frac{R_n-vn}{\sqrt{n}}\to c{\cal N}, 
 \end{align*}
where $c$ is some positive constant, ${\cal N}$ is the standard normal distribution,  
and the convergence is in the sence of distribution. 
For random walk in $\mathbb{Z}^{3}$ with mean $0$ and finite variance which satisfies 
the aperiodic condition:  
 \begin{align}\label{group}
\text{the group generated by the support of }X \text{ is all of }\mathbb{Z}^{d}, 
 \end{align}
the following results are shown by Jain and Pruitt \cite{Jain2, Jain4}:
 \begin{align*}
&E[(R_n-ER_n)^4]=O(n^2(\log n)^2),\\
&\mathrm{Var} R_n \sim cn \log n,\\
&\frac{R_n-vn}{\sqrt{n\log n}} \to c{\cal N},
 \end{align*}
where $c$ is a positive constant. 
The Law of large numbers for simple random walk in $\mathbb{Z}^{d}$ with $d\ge2$ has already shown by Dvoretzky and Erd\H{o}s \cite{Dvo1}. But an error in \cite{Dvo1} about $d=2$ was corrected by \cite{Jain1}. 
Also, for random walk in $\mathbb{Z}^{2}$ with mean $0$ and finite variance which satisfies  (\ref{group}) it was shown by Jain and Pruitt \cite{Jain1,Jain3} 
and Le Gall \cite{gall}:
 \begin{align*}
&ER_n=\pi\frac{n}{\log n}+O(\frac{n}{(\log n)^2}),\\
&\mathrm{Var} R_n\sim \frac{cn^2}{(\log n)^4},\\
&\lim_{n \to \infty}\frac{R_n}{ER_n}=1\quad {a.s.,}\\
&\frac{(\log n)^2}{n}(R_n-ER_n)\to -4\pi^2(det\Theta)\gamma(l),
 \end{align*}
where c is a positive constant, $l=\{(t,t')\in {\mathbb{R}}^2; 0\le t' \le t \le1\}$ and $\gamma(l)$ is the renormalized self-intersection local time of a planar Brownian motion 􏵔$\{W_t\}_{t\ge0}$, which is expressed formally by
$$\int\int_l \delta_0(W_t-W_s)dsdt-E[\int\int_l \delta_0(W_t-W_s)dsdt],$$
and $\Theta$ is the symmetric matrix satisfying $E􏵖(\theta,X_1)^2=(\theta, \Theta^2 \theta)$ for any $\theta\in {\mathbb{R}}^2$, where 􏵖$(·􏵐, ·􏵗)$ is the standard inner product on ${\mathbb {R}}^2$. 
The large deviations of $R_n$ were studied by Donsker and Varadhan \cite{Don1} and Hamana and Kesten \cite{Hama2}. 
In \cite{Hama2} it was shown that for any random walk in $\mathbb{Z}^{d}$ with $d\ge2$ which satisfies 
(\ref{group})
 \begin{align*}
\psi_0(x):=\lim_{n \to \infty}\frac{-1}{n}\log P(R_n\ge nx)\quad \text{exists}
 \end{align*}
for all $x$, and $\psi_0(\cdot)$ has the following properties:
 \begin{align*}
&\psi_0(x)=0 \quad \text{for } x\le v,\\
&0<\psi_0(x)<\infty \quad \text{for } v < x\le 1,\\
&\psi_0(x)=\infty \quad \text{for } 1< x,\\
& \psi_0 \text{ is continuous on }x\in[0,1],\\
& \psi_0 \text{ is convex on }x\in[0,1],\text{ and }\\
& \psi_0 \text{ is strictly increasing on }x\in[v,1].
 \end{align*}
Next we describe the known result about the multiple points of random walk range. 
Let $Q_n^{(p)}$ the number of the strictly $p$-multiple points of random walk range in the $n$ steps. 
That is,
 \begin{align*}
Q_n^{(p)}=\sharp \{S_i:0 \le i\le n,\sharp\{m:0\le m \le n,S_m=S_i \}=p \},
 \end{align*}
where $\sharp A$ denotes the cardinality of $A$. 
It was shown by Pitt \cite{pitt} that for any random walk and $p\ge1$
 \begin{align*}
\lim_{n \to \infty}\frac{Q_n^{(p)}}{n}
=v^2 (1-v)^{p-1}\quad a.s.,
 \end{align*}
and by Flatto \cite{Fla} that for simple random walk in $\mathbb{Z}^{2}$
 \begin{align*}
\frac{(\log n)^2Q_n^{(p)}}{n}\to \pi^2 \quad a.s..
 \end{align*}
Also, it was shown by Hamana \cite{Hama4,Hama5} that for random walk in $\mathbb{Z}^{2}$ with mean $0$ and finite variance which satisfies (\ref{group}) 
 \begin{align*}
&\frac{(\log n)^3}{n}(Q_n^{(p)}-EQ_n^{(p)})\to -16\pi^3(det\Theta)^2\gamma(l),\\
&\mathrm{Var} Q_n^{(p)}\sim c\frac{n^2}{(\log n)^6},
 \end{align*}
where c is a positive constant which is independent of $p$. 
In this paper, we deal with the inner boundary of random walk range.  
Let $L_n$ be the number of the inner boundary points of random walk range in the $n$ steps (see the next section for the definition). 
The lower bound of the expectation of the number of the inner boundary points is known by \cite{itai} Lemma $5$.  More precisely, it was shown that 
for simple random walk in $\mathbb{Z}^{d}$ with $d\ge2$ there exists a constant $C_d>0$  
such that for all $n\ge 1$,
 \begin{align*}
&EL_n \ge \frac{C_2n}{(\log n)^2} \quad  d=2,\\
 &EL_n \ge C_dn\quad d\ge3.
 \end{align*}
In \cite{itai}, it is noticed that the entropy of random walk is essentially governed by the size of the boundary of the trace. 
In this paper, we consider the asymptotic behavior of number of the inner boundary points 
and obtain analogues for $L_n$ of those results that are mentioned above.  
\section{Framework and Main Results }
\subsection{Framework}
We consider the random walk in $\mathbb{Z}^{d}$ with $d\ge1$ described in the introduction. 
Let $z$, $a$, $a_i$ $i\ge0$ points in ${\mathbb{Z}^{d}}$. 
A neighbor of $a$ is a point $z$ that satisfies dist$(a,z)=1$. 
Let ${\cal N}(a)$ the set of all neighbors of $a$:
$${\cal N}(a)=\{z : \text{dist}(a,z)=1 \}.$$
The inner boundary of random walk range $\{S_m\}_{m=0}^n$, denoted by $H_n$,  
is defined  by 
$$H_n=\{S_i :0 \le i\le n,\{S_m\}_{m=0}^n \not\supset {\cal N} (S_i) \}.$$ 
Let $L_n=\sharp H_n$. 
Also, we define  
\begin{align*}
J_n^p=\sharp &\{S_i\in H_n:\sharp\{m:0\le m \le n,S_m=S_i \}\ge p \},\\
J_n^{(p)}=\sharp &\{S_i\in H_n:\sharp\{m:0\le m \le n,S_m=S_i \}=p \}.
 \end{align*}
\subsection{Main Results }
  For each $i\ge1$ let $\{S^i_m\}_{m=0}^{\infty}$ be an independent copy of $\{S_m\}_{m=0}^{\infty}$, and define 
$T_{a,i}=\inf \{m\ge 1: S^i_m =a \}$ and $T_{a}=\inf \{m\ge 1: S_m =a \}$, the corresponding passage times. 
Let $\{S'_m\}_{m=0}^{\infty}$ denote an independent dual walk of $\{S_m\}_{m=0}^{\infty}$, namely an independent copy of  $\{-S_m\}_{m=0}^{\infty}$. 

\begin{thm}\label{conv1}
For any random walk in $\mathbb{Z}^{d}$ with $d\ge1$,
 \begin{align*}
\lim_{n\to\infty}\frac{L_n}{n} = q\quad a.s.,
 \end{align*}
where $q=P(\{S_m\}_{m=0}^{\infty}\cup\{S'_m\}_{m=0}^{\infty} \not\supset {\cal N} (0)$ and $0\notin\{S_m\}_{m=1}^{\infty}.)$ .

\end{thm}
\begin{thm}\label{conv1+}
For any random walk in $\mathbb{Z}^{d}$ with $d\ge1$,
\begin{align*}
\lim_{n\to\infty}\frac{J_n^{(p)}}{n} = &P(\{S_m\}_{m=0}^{\infty}\cup\{S'_m\}_{m=0}^{\infty} \cup(\cup_{i=1}^{p-1}\{S^i_m\}_{m=0}^{T_{0,i}}) \not\supset {\cal N} (0), \\
&0\notin \{S_m\}_{m=1}^{\infty}\cup \{S'_m\}_{m=1}^{\infty}
 \text{ and }0\in \{S^i_m\}_{m=1}^{\infty}
\text{ for }i=1,..,p-1.)
\quad a.s.,\\
\lim_{n\to\infty}\frac{J_n^{p}}{n} = &P(\{S_m\}_{m=0}^{\infty}\cup\{S'_m\}_{m=0}^{\infty} \cup (\cup_{i=1}^{p-1}\{S^i_m\}_{m=0}^{T_{0,i}}) \not\supset {\cal N} (0),\\
&0\notin \{S_m\}_{m=1}^{\infty}
 \text{ and }0\in \{S^i_m\}_{m=1}^{\infty}  
\text{ for }i=1,..,p-1.)
\quad a.s..
\end{align*}
\end{thm}
\begin{thm}\label{conv2}
For any random walk in $\mathbb{Z}^{d}$ with $d\ge2$ which satisfies (\ref{group}),
 \begin{align}\label{h}
\psi(x):=\lim_{n \to \infty}\frac{-1}{n}\log P(L_n\ge nx)\quad \text{exists}
 \end{align}
for all $x$, and $\psi(\cdot)$ has the following properties:
 \begin{align}\label{h1}
&\psi(x)=0 \quad \text{for } x\le q,\\
\label{h2}
&0<\psi(x)<\infty \quad \text{for } q < x\le 1,\\
\label{h3}
&\psi(x)=\infty \quad \text{for } 1< x,\\
\label{h4}
& \psi \text{ is continuous on }x\in[0,1],\\
\label{h5}
& \psi \text{ is convex on }x\in[0,1],\text{ and }\\
\label{h6}
& \psi \text{ is strictly increasing on }x\in[q,1].
 \end{align}
\end{thm}

 We call the random walk  $\{S_n\}_{n=0}^{\infty}$ simple if $P[S_1 =b_j]=1/2d$ where $b_j$, $j\in \{\pm1,\ldots. \pm d\}$ are neighbors of the origin in the square lattice $\mathbb Z^d$. 
\begin{thm}\label{conv3} ~  Let  $d=2$ and $p\ge1$ and suppose the random walk to be simple. Then
$$\lim_{n\to \infty}EL_n\times\frac{(\log n)^2}{n},
 \lim_{n\to \infty}EJ_n^{(p)}\times\frac{(\log n)^2}{n},
  \lim_{n\to \infty}EJ_n^{p}\times\frac{(\log n)^2}{n} 
$$ exist. Moreover, it holds that  
 \begin{align}
\label{s1}
\frac{{\pi}^2}{2} \le \lim_{n\to \infty}EL_n\times\frac{(\log n)^2}{n} &\le 2{\pi}^2,\\
\label{s2}
\frac{\tilde{c}^{p-1}\pi^2}{4} \le  \lim_{n\to \infty}EJ_n^{(p)}\times\frac{(\log n)^2}{n}  &\le \tilde{c}^{p-1}\pi^2,\\
\label{s3}
\frac{\tilde{c}^{p-1}\pi^2}{2}\le   \lim_{n\to \infty}EJ_n^{p}\times\frac{(\log n)^2}{n}  &\le2\tilde{c}^{p-1}\pi^2,
 \end{align}
where $\tilde{c}=P(T_0<T_b)$ with  some $b\in {\cal N}(0)$.
\end{thm}  


\section{Proof}
\subsection{Proof of Theorem \ref{conv1}} 
Let $\{Z_n\}_{n\in\mathbb{Z}}$ be a sequence of random variables defined by $Z_0=0$,
 $\{Z_n\}_{n=1}^{\infty}$=$\{S_n\}_{n=1}^{\infty}$, and   
$\{Z_{-n}\}_{n=1}^{\infty}$ =$\{S'_n\}_{n=1}^{\infty}$, 
where $\{S'_n\}_{n=1}^{\infty}$ is independent dual walk of $\{S_n\}_{n=1}^{\infty}$. 
We suppose $\{Z_n\}_{n\in{\mathbb {Z}}}$ to be a canonical realization, so that    $P$ is the probability measure on the product space $(\Pi_{n=-\infty}^\infty  \Om_n,  {\cal F})$ such that   $Z_n$ is the coordinate  map from  $\Om= \Pi_{n=-\infty}^\infty  \Om_n$ into $\Om_n$, where  $\Om_n $'s are copies of   $\mathbb Z^d$ and ${\cal F} = \sigma({\{Z_n}\}_{n\in \mathbb Z} )$.  
Let  $\phi$ be  the usual shift operator:  $\phi$ : $\Omega \to \Omega$ and 
 $Z_n\circ\phi=Z_{n+1}$.  
Let  $\phi ^m$ be the  $m$ times iterate of $\phi$: formally $\phi^0(\om) = \om$ and  $\phi^m = \phi \circ \phi^{m-1}$ $(m\geq 1)$.  Since $\phi$ is $P$-measure preserving, 
by the ergodic theorem  it holds that for any $A\in{\cal F}$
 \begin{align}\label{rewited}
\lim_{n\to\infty}\frac{1}{n} \sum_{m=0}^{n-1} 1_A(\phi^m\omega )= P(A) \quad a.s.
 \end{align}

\begin{proof}[Proof of Theorem \ref{conv1}]
Let $A$ be the event that $\{S_m\}_{m=0}^{\infty}\cup\{S'_m\}_{m=0}^{\infty} \not\supset {\cal N} (0)$ and $0\notin\{S_m\}_{m=1}^{\infty}$. 
In terms of $Z_m$, $A$ is expressed as $\{Z_m\}_{m\in \mathbb{Z} } \not\supset {\cal N} (Z_0)$ and $Z_0\notin \{Z_m\}_{m=1}^{\infty}$. 
Note that we can write
 \begin{align}\label{uu}
L_n=\sharp \{m: 0\le m \le n, \{S_{l}\}_{l=0}^{n}\not\supset {\cal N} (S_m) ,S_m \notin\{S_l\}_{l=m+1}^n \}.
 \end{align}
Then 
 \begin{align*}
L_n \ge \sum_{m=0}^{n} 1_{A}(\phi^m\omega )
 \end{align*}
since the right hand side equals 
$$\sharp \{m: 0\le m \le n, \{S_{l}\}_{l=0}^{\infty}\cup\{S'_{l}\}_{l=0}^{\infty} \not\supset {\cal N} (S_m) ,S_m \notin\{S_{m+l}\}_{l=1}^{\infty} \}.$$
Noting  that $A\in {\cal F}$, we apply  (\ref{rewited}) to see
  \begin{align}\label{upper}
\liminf_{n \to \infty} \frac{L_n}{n} \ge P(A) \quad a.s.
 \end{align}
To prove the inequality in opposite direction, 
let $A_k$ be the event that  $\{Z_m \}_{m=0}^k \cup\{Z_m\}_{m=-k}^{0} \not\supset {\cal N} (Z_0)$ and $Z_0 \notin\{Z_m\}_{m=1}^k$. 
Then, in view of  (\ref{uu}) we  obtain that for any $k\ge 1$
 \begin{align*}
L_n \le 2k+\sum_{m=k}^{n-k} 1_{A_k}(\phi^m\omega )
 \end{align*}
since the sum on the right hand side equals
$$\sharp \{m: k\le m \le n-k,\, \{S_{m+l}\}_{l= -k}^k \not\supset {\cal N} (S_m), \, S_m \notin\{S_{m+l}\}_{l=1}^k\}.$$
As before an application of  (\ref{rewited})  shows 
  \begin{align*}
\limsup_{n \to \infty} \frac{L_n}{n} \le P(A_k) \quad a.s..
 \end{align*}
Since $\cap_{k=1}^{\infty} A_k= A$, we now conclude
  \begin{align}\label{lower}
\limsup_{n \to \infty} \frac{L_n}{n} \le P(A) \quad a.s..
 \end{align}
By (\ref{upper}) and (\ref{lower}) the proof is complete.
\end{proof}

\begin{rem}
We can rewrite Theorem \ref{conv1} more generally. 
For any two finite sets $\tilde{H_j} \subset H \subset {\mathbb{Z}}^d$ for $j=1,2...N$, let 
$$L'_n=\sharp \bigcup_{j=1}^N\{S_i: 0\le i \le n, \{S_m\}_{m=0}^n \cap (S_i+H)= (S_i+\tilde{H_j}) \}.$$  
By the same argument as in the proof of Theorem \ref{conv1}, we can deduce that 
$$\lim_{n \to \infty}  \frac{L'_n}{n}= P((\{S_m\}_{m=0}^{\infty} \cup  \{S'_m\}_{m=0}^{\infty}) \cap H=\tilde{H_j}\text{ for some }j=1,2...N,
0\notin\{S_m\}_{m=1}^{\infty} ) \quad a.s..$$
\end{rem}

\begin{proof}[Proof of Theorem \ref{conv1+}]
First, we prove the upper bound of the first formula. 
Note that if $l=\{l_i\}_{i=1}^p$ and $G_{l}$ is the event that 
$$\{S_m\}_{m=0}^{\infty} \not\supset {\cal N} (S_{l_1}),  S_{l_1}=S_{l_2}=...=S_{l_p} \quad
\mbox{and} \quad  S_{l_1}\not\in (\{S_m\}_{m=0}^{\infty}-\{S_{l_i}\}_{i=1}^p),$$
 then  for any $n\ge (p-1)k$
 \begin{align*}
J_n^{(p)}\ge \sum_{l_1=0}^{n-(p-1)k} \sum_{l_2=l_1+1}^{l_1+k}
\sum_{l_3=l_2+1}^{l_2+k}...\sum_{l_p=l_{p-1}+1}^{l_{p-1}+k} 1_{G_l}(\omega).
 \end{align*}
So it holds that if $h=\{h_i\}_{i=2}^p$ and $\tilde{G}_{h}$ is the event that 
$\{Z_m\}_{m\in {\mathbb Z}} \not\supset {\cal N} (Z_{0})$,  $Z_0=Z_{h_2}=...=Z_{h_p}$ 
and $Z_{0}\not\in (\{Z_m\}_{m\in {\mathbb Z}}-(\{ Z_{h_i}\}_{i=2}^p\cup Z_0) )$,
then 
 \begin{align*}
J_n^{(p)}\ge \sum_{l_1=0}^{n-(p-1)k} \sum_{h_2=1}^{k}
\sum_{h_3=h_2+1}^{h_2+k}...\sum_{h_p=h_{p-1}+1}^{h_{p-1}+k} 1_{\tilde{G}_h}(\phi^{l_1}\omega).
 \end{align*}
Noting  that $\tilde{G}_{h} \in {\cal F}$, by (\ref{rewited}) we get for any $k\ge1$
 \begin{align}\label{d1}
\liminf_{n\to \infty} \frac{J_n^{(p)}}{n} 
\ge P&(T^{i}_{Z_0}-T^{i-1}_{Z_0}\le k\text{ for }i=1,..,p-1, T^p_{Z_0}=\infty,\\
\notag
 &\{Z_m\}_{m\in {\mathbb Z}} \not\supset {\cal N}( Z_{0})
\text{ and } Z_0\notin\{Z_m\}_{m=-\infty}^{-1} .),
 \end{align}
where $T^j_{a}=\inf\{m>T^{j-1}_{a}: Z_m=a\}$ and $T_a^0=0$.
Therefore, by the strong Markov property we get
 \begin{align}\notag
\liminf_{n\to \infty} \frac{J_n^{(p)}}{n} 
\ge
 &P(T^{p-1}_{Z_0}<\infty, T^{p}_{Z_0}=\infty, \{Z_m\}_{m\in {\mathbb Z}} \not\supset {\cal N} (Z_{0})\text{ and }Z_0 \notin\{Z_m\}_{m=-\infty}^{-1} .)\\
\notag
=&P(\{S_k\}_{k=0}^{\infty}\cup\{S'_k\}_{k=0}^{\infty} \cup(\cup_{i=1}^{p-1}\{S^i_k\}_{k=0}^{T_{0,i}}) \not\supset {\cal N}(0),\\
\label{d2}
&0 \notin\{S_k\}_{k=1}^{\infty}\cup\{S'_k\}_{k=1}^{\infty}
 \text{ and }0\in \{S^i_k\}_{k=1}^{\infty} 
\text{ for }i=1,..,p-1.)
\quad a.s..
 \end{align}
By (\ref{d1}) and (\ref{d2}) we get the one side inequality.
To prove the inequality in opposite direction, note that 
 \begin{align*}
J_n^{(p)}\le 2k+\sum_{l_1=k}^{n} \sum_{l_2=l_1+1}^{\infty}...\sum_{l_p=l_{p-1}+1}^{\infty} 1_{G_{l,k}}(\omega),
 \end{align*}
where $G_{l,k}$ is the event that 
$\{S_m\}_{m=l_1-k}^{l_p+k} \not\supset {\cal N}(S_{l_1})$, $S_{l_1}=S_{l_2}=...=S_{l_p}$, 
and $S_{l_1}\not\in (\{S_m\}_{m=l_1-k}^{l_p+k}-\{S_{l_i}\}_{i=1}^p)$. 
Hence,  if we set $h=\{h_i\}_{i=2}^p$, then
 \begin{align*}
J_n^{(p)}\le 2k+\sum_{l_1=k}^{n} \sum_{h_2=1}^{\infty}...\sum_{h_p=h_{p-1}+1}^{\infty} 1_{\tilde{G}_{h,k}}(\phi^{l_1}\omega),
 \end{align*}
where $\tilde{G}_{h,k}$ is the event that 
$\{Z_m\}_{m=-k}^{h_p+k}\not\supset {\cal N} (Z_{0})$, $Z_0=Z_{h_2}=...=Z_{h_p}$, 
and $Z_{0}\not\in (\{Z_m\}_{m=-k}^{h_p+k}-(\{ Z_{h_i}\}_{i=2}^p\cup Z_0) )$. 
If we note that $\tilde{G}_{h,k} \in {\cal F}$, by (\ref{rewited}) we get for any $k\ge1$
 \begin{align}\notag
&\limsup_{n\to \infty} \frac{J_n^{(p)}}{n} \\
\label{d3}
\le &P(T^{p-1}_{Z_0}<\infty,T^{p}_{Z_0}-T^{p-1}_{Z_0}>k,\{Z_m\}_{m=-k}^{T^{p-1}_{Z_0}+k} \not\supset {\cal N}( Z_{0}) \text { and }Z_0\notin\{Z_m\}_{m=-k}^{-1}.)\quad a.s..
 \end{align}
By the monotonicity in $k$, we find that as $k\to \infty$ the right hand side of the last formula converges to
 \begin{align}\label{d4}
P(T^{p-1}_{Z_0}<\infty, T^p_{Z_0}=\infty, \{Z_m\}_{m\in {\mathbb Z}} \not\supset {\cal N} (Z_{0})\text{ and } Z_0\notin\{Z_m\}_{m=-\infty}^{-1}.).
 \end{align}
By (\ref{d1}), (\ref{d2}), (\ref{d3}) and (\ref{d4}), the proof of the first formula is complete. 
Next we prove the second formula. 
Note that it holds that  for any $n\ge (p-1)k$
 \begin{align*}
J_n^{p}&\ge \sum_{l_1=0}^{n-(p-1)k} \sum_{l_2=l_1+1}^{l_1+k}
\sum_{l_3=l_2+1}^{l_2+k}...\sum_{l_p=l_{p-1}+1}^{l_{p-1}+k} 1_{\overline{G}_l}(\omega),\\
J_n^{p}&\le 2k+\sum_{l_1=k}^{n} \sum_{l_2=l_1+1}^{\infty}...\sum_{l_p=l_{p-1}+1}^{\infty} 1_{\overline{G}_{l,k}}(\omega),
 \end{align*}
where $\overline{G}_{l}$ is the event that 
$\{S_m\}_{m=0}^{\infty} \not\supset {\cal N} (S_{l_1})$, $S_{l_1}=S_{l_2}=...=S_{l_p}$ 
and $S_{l_1}\not\in (\{S_m\}_{m=l_1+1}^{\infty}-\{S_{l_i}\}_{i=2}^p)$, 
and $\overline{G}_{l,k}$ is the event that 
$\{S_m\}_{m=l_1-k}^{l_p+k} \not\supset {\cal N}(S_{l_1})$, $S_{l_1}=S_{l_2}=...=S_{l_p}$  
and $S_{l_1}\not\in (\{S_m\}_{m=l_1+1}^{l_p+k}-\{S_{l_i}\}_{i=2}^p)$. 
Since the rest of proof of the second formula is the same as the first one, we omit it.
\end{proof}
\subsection{Proof of Theorem \ref{conv2}}
\begin{lem}\label{sub+}
Let $c=d+2d^2+8d(2d+1)^2$. 
Then, there exists constant $\zeta\in (0,1)$ (depending only on $d$) such that for all integer $n,m\ge0$
and $y,z \in [0,\infty)$, it holds that
 \begin{align*}
P(L_{n+m}\ge y+z-c(nm)^{\frac{1}{d+1}})
\ge \frac{1}{2}\zeta^{{d(nm)}^\frac{1}{d+1}+d} P(L_n\ge y)P(L_m\ge z).
 \end{align*}
\end{lem}
\begin{proof}
Let $\hat{X}_1,\hat{X}_2$,... be an independent copy of $X_1,X_2$,..., $\hat{S}_0=0$, $\hat{S}_k=\sum_{i=1}^{k}\hat{X}_i$, and $\hat{L}_n$ be 
the  number of the inner boundary points of $\{\hat{S}_0$, $\hat{S}_1$,...,$\hat{S}_n\}$, that is,
$$\hat{L}_n=\sharp \{\hat{S}_i :0 \le i\le n,\{\hat{S}_m\}_{m=0}^n \not\supset {\cal N} (\hat{S}_i) \}.$$
We define $L\{a_1,...,a_l\}$ to be the cardinality of the inner boundary of $\{a_i\}_{i=1}^l$, i.e., 
 \begin{align*}
L\{a_1,...,a_l\}=\sharp \{a_i : 1\le i\le l,\{a_k\}_{k=1}^l \not\supset {\cal N} (a_i )\},
 \end{align*}
and $U\{a_1,...,a_l\}$ to be the union of the outer boundary and the inner boundary of the range of $\{a_i\}_{i=1}^l$, i.e., 
 \begin{align*}
U\{a_1,...,a_l\}=&\{a_i : 1\le i\le l,\{a_k\}_{k=1}^l \not\supset {\cal N} (a_i) \}\\
\cup&\{x\in{\mathbb{Z}^{d}} : x \not\in\{a_k\}_{k=1}^l 
\text{ and there exists }y \in\{a_k\}_{k=1}^l \text{ such that dist}(x,y)=1  \}.
 \end{align*}
Moreover, we define
 \begin{align*}
U[a,b]=U\{S_a,S_{a+1},...,S_b\},\quad
\hat{U}[a,b]=U\{\hat{S}_a,\hat{S}_{a+1}...,\hat{S}_b\}.
 \end{align*}
Next we difine for $\lambda \in {\mathbb{Z}^{d}}$
$$N_{n,m}(\lambda)=
\sharp \{u\in \mathbb{Z}^d : u\in U[0,n] \text{ and }u\in S_n+\lambda+\hat{U}[0,n]\}.$$
For any fixed integers $p\ge0$ and $n\ge0,$ consider the random walk defined by
\begin{equation*}
T_k=\begin{cases}
S_k &(k\le n+p)
\\
S_{n+p}+\hat{S}_{k-n-p}&(k>n+p).
\end{cases}
\end{equation*}
Of course, $\{T_k\}_{k=0}$ has the same distribution as $\{S_k\}_{k=0}$, 
and hence also 
$P(L_{n+p+m}\ge l)=P(L\{T_0,...,T_{n+p+m-1}\}\ge l)$.  
We claim that on the event
 \begin{align}\label{for1}
\{S_{n+p}-S_n=\lambda \},
 \end{align}
it holds that 
 \begin{align}\label{for2}
L\{T_0,...,T_{n+p+m-1}\}\ge L_n+\hat{L}_m-N_{n,m}(\lambda),
 \end{align}
and 
 \begin{align*}
N_{n,m}(\lambda)=N_{n,m}(T_{n+p}-T_n)
= \sharp U\{T_0,...,T_{n-1}\} \cap U\{T_{n+p},...,T_{n+p+m-1} \}.
 \end{align*}
Owing to the assumption (\ref{group}),
 we can pick $d$ linearly independent vectors $v_1,..,v_d \in \mathbb{Z}^d$ 
for which $P(X=v_i)>0.$ We can then choose $0<\zeta<1$ such that $P(X=v_i)\ge \zeta$. 
We set
 \begin{align*}
\Xi_q=\{ \sum_{i=1}^d k_iv_i : 0\le k_i \le q \} \subset \mathbb{Z}^d.
 \end{align*}
For any $\lambda=\sum_{i=1}^d k_iv_i \in \Xi_q$, 
we then have that for $p=p(\lambda)=\sum_{i=1}^d k_i \le dq$,
 \begin{align*}
P(S_{n+p}-S_n=\lambda)=P(S_p=\lambda) \ge \zeta^p \ge \zeta^{dq}.
 \end{align*}
Moreover,
 \begin{align*}
\sharp \Xi_q=(\text{number of vectors }\omega \in \Xi_q) =(q+1)^d.
 \end{align*}
We take
 \begin{align*}
q=q(n,m)= \lceil (nm)^{ \frac{1}{d+1}} \rceil,
 \end{align*}
where $\lceil a \rceil$ denotes the smallest integer $\ge a$. 
Note that the simple monotonicity in $n$ of $P(L_n\ge y)$ does not hold, that is, it does not hold for any $n$, $y$, $v>0$ 
$P(L_{n+v}\ge y) \ge P(L_n\ge y)$. But it holds that for any $n$, $y$, $v>0$ 
 \begin{align}\label{note2}
P(L_{n+v}\ge y-2dv) \ge P(L_n\ge y).
 \end{align}
As a result of (\ref{for2}) and (\ref{note2}) for each $\lambda \in \Xi_q$
 \begin{align*}
&P(L_{n+m}\ge y+z-c(nm)^{\frac{1}{d+1}})\\
\ge &P(L_{n+dq+m}\ge y+z-(c-d)(nm)^{\frac{1}{d+1}})\\
\ge &P(L_{n+p+m}\ge y+z-(c-d-2d^2)(nm)^{\frac{1}{d+1}})\\
\ge &P(L_{n}\ge y,\hat{L}_m\ge z, 
S_{n+p}-S_n=\lambda,
N_{n,m}(\lambda)\le \frac{1}{4d}(c-d-2d^2)(nm)^{\frac{1}{d+1}}).
 \end{align*}
The event (\ref{for1}) depends only on $X_i$ with $n<i\le n+p$, and is independent of the events $\{L_{n}\ge y\}, \{\hat{L}_m\ge z\}$ and of random variable $N_{n,m}(\lambda)$. Consequently,
 \begin{align*}
&P(L_{n+m}\ge y+z-c(nm)^{\frac{1}{d+1}})\\
\ge &P(S_{n+p}-S_n=\lambda)P(L_{n}\ge y,\hat{L}_m\ge z, 
N_{n,m}(\lambda)\le  \frac{1}{4d}(c-d-2d^2)(nm)^{\frac{1}{d+1}})\\
\ge &\zeta^{dq} P(L_{n}\ge y,\hat{L}_m\ge z, 
N_{n,m}(\lambda)\le  \frac{1}{4d}(c-d-2d^2)(nm)^{\frac{1}{d+1}}).
 \end{align*}
Since this inequality holds for all $\lambda \in \Xi_q$, we can take its average over $\Xi_q$ to obtain
 \begin{align}\notag
&P(L_{n+m}\ge y+z-c(nm)^{\frac{1}{d+1}})\\
\notag
\ge &\frac{\zeta^{dq}}{|\Xi_q|}
\sum_{\lambda \in \Xi_q} P(L_{n}\ge y,\hat{L}_m\ge z, 
N_{n,m}(\lambda)\le  \frac{1}{4d}(c-d-2d^2)(nm)^{\frac{1}{d+1}})\\
\label{for3}
= &\frac{\zeta^{dq}}{|\Xi_q|}
E[ \sharp \{\lambda \in \Xi_q: N_{n,m}(\lambda)\le  \frac{1}{4d}(c-d-2d^2)(nm)^{\frac{1}{d+1}}\}
I_{\{L_{n}\ge y\}}I_{\{\hat{L}_m\ge z\}}].
 \end{align}
We shall shortly show that for all integer $n,m\ge0$
 \begin{align}\label{promise}
\sharp \{\lambda \in \Xi_q: N_{n,m}(\lambda)\le  \frac{1}{4d}(c-d-2d^2)(nm)^{\frac{1}{d+1}}\}\ge \frac{1}{2}(q+1)^d .
 \end{align}
Taking this for granted and recalling that 
 $\{L_{n}\ge y\}$ and $\{\hat{L}_m\ge z\}$ are independent,  if (\ref{promise}) is true, then we infer from (\ref{for3}) that 
 \begin{align}
\notag
&P(L_{n+m}\ge y+z-c(nm)^{\frac{1}{d+1}})\\
\notag
\ge &P(L_{n+dq+m}\ge y+z-(c-d)(nm)^{\frac{1}{d+1}})\\
\notag
\ge &\frac{\zeta^{dq}}{(q+1)^d}\frac{1}{2}(q+1)^d P(L_{n}\ge y)P(L_{m}\ge z)\\
\label{go}
\ge &\frac{1}{2} \zeta^{dq}P(L_{n}\ge y)P(L_{m}\ge z),
 \end{align}
which implies the inequality of the lemma. 
It remains to prove (\ref{promise}). We have
 \begin{align*}
&\sum_{\lambda \in \Xi_q}N_{n,m}(\lambda)
\le \sum_{\lambda\in \mathbb{Z}^d}N_{n,m}(\lambda) \\
=&\sum_{u \in \Xi_q}\sum_{\lambda \in \mathbb{Z}^d}
I[u\in U\{S_0,S_1,...,S_{n-1}\}] \times I[u\in U\{S_n+\lambda+\{\hat{S}_0,\hat{S}_1,...,\hat{S}_{n-1} \}\} ] \\
=&\sum_{u \in \Xi_q}I[u\in U\{S_0,S_1,...,S_{n-1} \}]\times
\sum_{\lambda \in \mathbb{Z}^d}
I[\lambda \in U\{S_n+u+\{\hat{S}_0,\hat{S}_1,...,\hat{S}_{n-1} \} \} ] \\
=&\sum_{u \in \Xi_q} I[u\in U\{S_0,S_1,...,S_{n-1}\}] \times 
\sharp U\{u-S_n-\hat{S}_0, u-S_n-\hat{S}_1,...,u-S_n-\hat{S}_{m-1}\} \\
=&\sum_{u \in \Xi_q} I[u\in U\{S_0,S_1,...,S_{n-1} \}] \times \sharp \hat{U}_m\\
\le &\sharp U_n\sharp \hat{U}_m\le (2d+1)^2nm.
 \end{align*}
So 
it holds that
 \begin{align*}
&\sharp \{\lambda \in \Xi_q: N_{n,m}(\lambda)\ge  \frac{1}{4d}(c-d-2d^2)(nm)^{\frac{1}{d+1}}\}\\
&\le \sum_{\lambda\in \Xi_q} \frac{N_{n,m}(\lambda)}{2(2d+1)^2(nm)^{\frac{1}{d+1}} }
\le \frac{(2d+1)^2nm}{2(2d+1)^2(nm)^{\frac{1}{d+1}} }
=\frac{1}{2}(nm)^{\frac{d}{d+1}}
\le \frac{1}{2}\sharp \Xi_q,
 \end{align*}
and hence 
 \begin{align*}
\sharp \{\lambda \in \Xi_q: N_{n,m}(\lambda)\le  \frac{1}{4d}(c-d-2d^2)(nm)^{\frac{1}{d+1}}\}\ge \frac{1}{2}(q+1)^d .
 \end{align*}
The proof of Lemma \ref{sub+} is completed.
\end{proof}
For $x\in {\mathbb{R}}$ we define
 \begin{align}\label{def1}
\psi(x)=\liminf_{n\to\infty} \frac{-1}{n} \log P(L_n \ge nx).
 \end{align}
Observe that $\psi(x)$ is nondecreasing in $x$. 
Moreover, it is bounded on $[0,1]$ because by (\ref{group}) there exists $a\in {\mathbb Z}\setminus \{0\}$ such that
 \begin{align*}
&P(L_n\ge n)\ge P(X_1(i)=X_2(i)=...=X_n(i)\neq0) =[P(X_1(i)=a)]^n
 \end{align*}
and $P(X_1(i)=a)>0$ for some $1\le i \le d$,  
where $X_j(i)$ denotes the $i$-th component of $X_j$. 
Hence,  we find 
 \begin{align}\label{g+}
\psi(1)<\infty.
 \end{align}  
We have to prove that liminf in (\ref{def1}) can be replaced by lim. 
We first show that this is permissible for any $x\in[0,1)$ at which $\psi$ is continuous from the right. 
\begin{prop}\label{right cont}
If (\ref{group}) holds and $d\ge2$ and if $\psi$ is right continuous at a given $x\in[0,1)$, then 
 \begin{align*}
\psi(x)=\lim_{n \to \infty}\frac{-1}{n}\log P(L_n\ge nx).
 \end{align*}
\end{prop}
\begin{proof}
Since the idea of this proof is the same as in \cite{Hama2}, Proposition $2$,  we only give an  outline of the proof. 
Owing to  Lemma \ref{sub+} we can choose a constant $1<c<\infty$ so  that
 \begin{align}\label{sub++}
P(L_{n+m+dq}\ge y+z-c(nm)^{\frac{1}{d+1}})
\ge \frac{1}{2}\zeta^{{d(nm)}^\frac{1}{d+1}+d} P(L_n\ge y)P(L_m\ge z),
 \end{align}
where $q= \lceil (nm)^{ \frac{1}{d+1}} \rceil$. 
We set $\eta=\frac{d-1}{d+1}$ and $\xi=\frac{2}{d+1}$. 
If we define for any integer $N\ge1$,
 \begin{align*}
\sigma(0)&=N,\\
\sigma(k+1)&=2\sigma(k)+d\lceil [\sigma(k)]^{\xi} \rceil \quad
 k\ge0,
 \end{align*}
the following holds:
 \begin{align}\label{e1}
\frac{\sigma(i-1)}{\sigma(i)}\le \frac{1}{2}, \quad
\sigma(i)\ge 2^iN,
 \end{align}
and for some constants $c_1$, $c_2$, $N_0<\infty$ and $N\ge N_0$
 \begin{align}\label{e2}
1\le \frac{\sigma(k)}{2^kN}&\le 1+\frac{c_1}{N^{\eta}} \le 2\\
\label{e2+}
\sum_{i=0}^{k-1}2^{k-i}[\sigma(i)]^{\xi} &\le c_2N^{-\eta}\sigma(k).
 \end{align}
Now let $x\in[0,1)$ be such that $\psi$ is right continuous at $x$ and let $\epsilon>0$. 
Take $\delta\in(0,1)$ such that
 \begin{align*}
\psi(x+4\delta)\le \psi(x)+\epsilon.
 \end{align*}
Take $c_3 =\frac{\zeta^d}{2}<1$ and fix $l\ge2$ such that
 \begin{align}\label{e3}
(1-2^{-l+2})(x+2\delta)\ge x+\delta,\\
\label{e3+}
\frac{\delta}{2}>2d2^{-l+2}.
 \end{align}
Finally, fix $N\ge N_0$ so that 
 \begin{align}
\notag
P(L_n \ge N(x+4\delta))&\ge \exp[-N(\psi(x+4\delta)+\epsilon)]\\
\notag &\ge \exp[-N(\psi(x)+2\epsilon)],\\
\label{rew+}
1+\frac{c_1}{N^{\eta}} &\le \frac{x+4\delta}{x+3\delta}\\
\label{rew++}
N^{-\eta}&<\min\bigg\{\frac{\delta}{cc_2}, \frac{-2\epsilon}{c_2d\log \zeta}, \frac{1}{2d} \bigg\},\\
\label{rew+++}
5cl(3d+2)(N^{\xi}+1)&<\frac{1}{2}\delta N,\\
\label{rew++++}
 \text {and } \quad \frac{2}{N}|\log c_3| &<\epsilon.
 \end{align}
We shall first consider $P(L_n \ge nx)$ for $n\in \{\sigma(k) \} _{k=0}$. 
If we set  $m=n=\sigma(k-1)$ and 
 \begin{align*}
y=z=2^{k-1}N(x+4\delta)- c\sum_{i=0}^{k-2}2^{k-1-i}[\sigma(i)]^{\xi},
 \end{align*}
then (\ref{sub++}) gives for $k\ge1$ 
 \begin{align}
\notag &P(L_{\sigma(k)}\ge 2^k N(x+4\delta)- c\sum_{i=0}^{k-1}2^{k-i}[\sigma(i)]^{\xi})\\
\label{pre} \ge &c_3 \zeta^{d[\sigma(k-1)]^{\xi}}
[P( L_{\sigma(k-1)}\ge 2^{k-1} N(x+4\delta)- c\sum_{i=0}^{k-2}2^{k-1-i}[\sigma(i)]^{\xi})]^2.
 \end{align}
 By (\ref{e2}), (\ref{e2+}), (\ref{rew+}) and (\ref{rew++})  we also have
 \begin{align}\label{pre+}
2^k N(x+4\delta)- c\sum_{i=0}^{k-1}2^{k-i}[\sigma(i)]^{\xi}
 \ge \sigma(k)(x+2\delta).
 \end{align}
Hence, by (\ref{rew++}), (\ref{pre}), (\ref{pre+}) we get
 \begin{align}\label{pre++}
P(L_{\sigma(k)} \ge \sigma(k)(x+2\delta))
\ge[c_3]^{2^{k+1}} \exp [-2^kN(\psi(x)+3\epsilon)].
 \end{align}
Next we  expand  $n$ into a linear combination of the $\sigma(k)$ in the same as in \cite{Hama2}, Proposition $2$. Recall that we have fixed $l$ in (\ref{e3}) and  (\ref{e3+}). 
Now let $n\ge \sigma(2l)$, and take
 \begin{align*}
\hat{n}=n-2dl\lceil n^{\xi} \rceil.
 \end{align*} 
Owing to  (\ref{e2}) and (\ref{rew++}) we can pick $k_r$, $\alpha_r \in\{1,2\}$, $p\le l$ such that
 \begin{align}\label{pre+++}
0\le \hat{n}- \sum_{i=1}^{p} \alpha_i \sigma(k_i)<2^{-l+2}n.
 \end{align}

We set $\beta:=\sum_{i=1}^p \alpha_i$ and let $n_1<n_2<...< n_{\beta}$ 
be  number of the form $\sum_{i=1}^{j} \alpha_i \sigma(k_i)$ or
$\sum_{i=1}^{j} \alpha_i \sigma(k_i)-\sigma(k_j)$; 
the latter form is included only if $\alpha_j=2$. 
We now apply (\ref{sub++}) with  
$y=n_{\gamma}(x+2\delta)-5c\gamma n^{\xi}$,  
$z=(n_{\gamma+1}-n_{\gamma})(x+2\delta)$, 
$n=n_{\gamma}+d\gamma\lceil n^{\xi} \rceil$ and 
$m=n_{\gamma+1}-n_{\gamma}$. 
Using  (\ref{note2}) and (\ref{sub++}) we then find for $c>1$
 \begin{align}
\notag &P(L_{n_{\gamma+1}+d(\gamma+1)\lceil n^{\xi} \rceil}
\ge n_{\gamma+1}(x+2\delta)-5c(\gamma+1) n^{\xi})\\
\label{e5}
\ge&\frac{1}{2}\zeta^{dn^{\xi}+d} 
P(L_{n_{\gamma}+d\gamma\lceil n^{\xi} \rceil}
\ge n_{\gamma}(x+2\delta)-5c\gamma n^{\xi})
\times P(L_{n_{\gamma+1}-n_{\gamma}} \ge (n_{\gamma+1}-n_{\gamma})(x+2\delta)).
 \end{align}
Consequently,  by (\ref{pre}) and (\ref{e5}) we get
 \begin{align}\notag
&P(L_{n_{\beta}+d\beta\lceil n^{ \xi} \rceil} \ge n_{\beta}(x+2\delta)-5c\beta n^{\xi})\\
\label{rew1}
\ge& [c_3]^{2l+\sum_{j=1}^p \alpha_j 2^{k_j+1} } \zeta^{2ldn^{\xi}} 
\exp [-\sum_{j=1}^p \alpha_j 2^{k_j} N(\psi(x)+3\epsilon)].
 \end{align}
Now we apply (\ref{e3}), (\ref{rew+++}) and (\ref{pre+++}) to see that 
 \begin{align}\label{rew2}
n_{\beta}(x+2\delta)-5c\beta n^{\xi} \ge n(x+\frac{\delta}{2}). 
 \end{align}
On the other hand,  by (\ref{e3+}) and (\ref{pre+++}) we get for sufficiently large $n$,
 \begin{align}\label{rewr++}
\frac{\delta}{2}n 
\ge 2d(2^{-l+2}n+d(2l-\beta) \lceil n^{ \xi} \rceil)
\ge 2d(n-n_{\beta}-d\beta \lceil n^{ \xi} \rceil). 
 \end{align}
Hence,   by (\ref{note2}) and (\ref{rewr++}) we get for sufficiently large $n$,
 \begin{align}\label{rew3}
 P(L_n\ge nx)
\ge P(L_{n_{\beta}+d\beta\lceil n^{ \xi} \rceil} \ge n(x+\frac{\delta}{2})). 
 \end{align}
Since $\sum_{j=1}^p \alpha_j 2^{k_j}<n$, by (\ref{rew++++}), (\ref{rew1}), (\ref{rew2}) and (\ref{rew3}) we get the assertion of the proposition.
\end{proof}
\begin{lem}\label{dd}
For $d\ge2$, $\psi$ is convex and continuous on $(0,1)$.  
\end{lem}
\begin{proof}
To prove the convexity on continuous points of $\psi$, we can apply Lemma \ref{sub+}. 
The proof that $\psi$ is continuous on $(0,1)$ is the same as in \cite{Hama2}, Lemma $3$. 
The details are omitted. 
\end{proof}
\begin{proof}[Proof of Theorem \ref{conv2}]
It is obvious that (\ref{h3}) holds. 
To prove that $\psi$ is continuous at $0$, 
note that  $\psi(x)=0$ for $x\le0$, while by (\ref{group}) there exists $a\in{\mathbb Z} \setminus \{0\}$ such that  for sufficiently small $\delta\in(0,1)$, 
 \begin{align*}
P(L_n\ge \delta n)\ge P(X_1=X_2,...., X_{\lceil \delta n\rceil}\neq 0) 
\ge [P(X_1(i)=a)]^{\lceil \delta n\rceil}.
 \end{align*} 
It follows that $\psi(\delta)\le -\delta \log P(X_1(i)=a)$ as in (\ref{g+}), 
hence, also $\lim_{\delta \to 0} \psi(\delta)=0$. 
Then, the proof that $\psi$ is continuous at $1$ is the same as in \cite{Hama2}, Proposition $4$, and combined with  Lemma \ref{right cont}  and (\ref{h4}) this continuity shows (\ref{h}).  

Now that we have continuity of $\psi$ on $[0,1]$,  we obtain the convexity of $\psi$ on $[0,1]$ from lemma \ref{dd}. 
We also have continuity of $\psi$ at $q$, so that also  (\ref{h1}) holds. 

We can show  that the right derivative at $\eta=0$  of $\lim_{n \to \infty} \frac{-1}{n}\log Ee^{\eta L_n}$ is $q$ by the argument given in \cite{Hama3}, hence 
we get  (\ref{h2}). 
Also, the proof of (\ref{h6}) is the same as in \cite{Hama2}, Proposition $4$.  
\end{proof}
\subsection{Proof of Theorem \ref{conv3}}
In this subsection we consider the simple random walk in $\mathbb{Z}^{2}$. 
We denote the neighbors of $0$ by $b_1,...,b_4$. 
In the following lemma,  $a_n \sim c_n$ means $\frac{a_n}{c_n} \to 1$ $(n\to \infty)$ for sequences $a_n$ and $c_n$,.
\begin{lem}\label{local}
For any $i$,
  \begin{align*}
P^{b_i}(\{S_m\}_{m=1}^n \cap \{0,b_{i}\} = \emptyset)
+P^{0}(\{S_m\}_{m=1}^n \cap \{0,b_{i}\} = \emptyset)
\sim \frac{\pi}{\log n}. 
 \end{align*}
In particular, by symmetry of the roles played by  $0$ and $b_i$, we have
  \begin{align*}
P^{b_i}(\{S_m\}_{m=1}^n \cap \{0,b_{i}\} = \emptyset)
=P^{0}(\{S_m\}_{m=1}^n \cap \{0,b_{i}\} = \emptyset) 
\sim \frac{\pi}{2\log n}. 
 \end{align*}
\end{lem}

\begin{rem}
While this lemma has been already proven by using  Corollary $2$ and $(1.2)$ in \cite{SP2}, we give a direct (and hence simpler) proof. 
\end{rem}
 \begin{proof}
Let $$\gamma(n)=P^{b_i}(\{S_m\}_{m=1}^{2n} \cap \{0,b_{i}\} = \emptyset)
+P^{0}(\{S_m\}_{m=1}^{2n} \cap \{0,b_{i}\} = \emptyset).$$ 
If we consider the last return time to the set  $\{0, b_i\}$  in the first $2n$ steps, we get
  \begin{align}\notag
1 = &\sum_{k=0}^{n} P(S_{2k}=0)P^{0}(\{S_m\}_{m=1}^{2n-2k}\cap \{0,b_{i}\} = \emptyset )\\
\label{local+}
+&\sum_{k=0}^{n-1} P(S_{2k+1}=b_i)P^{b_i}(\{S_m\}_{m=1}^{2n-2k-1}\cap \{0,b_{i}\} = \emptyset).
 \end{align}
We first show the upper bound.   
By local central limit theorem (cf., for example, Theorem $1.2.1$ in \cite{Law}), it holds that for each $i$,
 \begin{align}\label{local2}
P(S_{2k} = 0) \sim \frac{1}{\pi k},\quad P(S_{2k+1} = b_i) \sim \frac{1}{\pi k}, \quad \text{when }k\to \infty.
 \end{align}
So we can rewrite (\ref{local+}) as
  \begin{align*}
1 = &P^0(\{S_m\}_{m=1}^{2n}\cap \{0,b_{i}\} = \emptyset)\\
+&\sum_{k=1}^{n} \frac{1}{\pi k}P^0(\{S_m\}_{m=1}^{2n-2k}\cap \{0,b_{i}\} = \emptyset)(1+o(1))\\
+&P^{b_i}(\{S_m\}_{m=1}^{2n-1}\cap \{0,b_{i}\} = \emptyset)\\
+&\sum_{k=1}^{n-1} \frac{1}{\pi k}P^{b_i}(\{S_m\}_{m=1}^{2n-2k-1}\cap \{0,b_{i}\} = \emptyset)(1+o(1)).
 \end{align*}
Since $\gamma(n)$ is nonincreasing, it holds that 
 \begin{align*}
1 \ge\gamma(n)+ \sum_{k=1}^{n} \frac{1}{\pi k}\gamma(n)(1+o(1)).
 \end{align*}
So we get 
  \begin{align}\label{form1}
\gamma(n) \le \frac{\pi}{\log n}(1+o(1)).
 \end{align}
Next we show the lower bound. 
For any $0\le l \le n$, it holds that
 \begin{align*}
1 \le 
&\sum_{k=0}^{l} P(S_{2k}=0)P^{0}(\{S_m\}_{m=1}^{2n-2k}\cap \{0,b_{i}\} = \emptyset)\\
+&\sum_{k=0}^{l} P(S_{2k+1}=b_i)P^{b_i}(\{S_m\}_{m=1}^{2n-2k-1}\cap \{0,b_{i}\} = \emptyset)\\
+&\sum_{k=l+1}^{n} P(S_{2k}=0)
+\sum_{k=l+1}^{n-1} P(S_{2k+1}=b_i)\\
  \le 
&\sum_{k=0}^{l} P(S_{2k}=0)P^{0}(\{S_m\}_{m=1}^{2n-2l}\cap \{0,b_{i}\} = \emptyset)\\
+&\sum_{k=0}^{l} P(S_{2k+1}=b_i)P^{b_i}(\{S_m\}_{m=1}^{2n-2l-1}\cap \{0,b_{i}\} = \emptyset)\\
+&\sum_{k=l+1}^{n} P(S_{2k}=0)
+\sum_{k=l+1}^{n-1} P(S_{2k+1}=b_i).
 \end{align*}
Again by(\ref{local2}), it holds that
  \begin{align*}
1 \le \gamma(n-l+1) \frac{\log n}{\pi}(1+o(1))
+\frac{2}{\pi} \log\frac{n}{l}(1+o(1)).
 \end{align*}
If we pick $l=n-\lceil \frac{n}{\log n}\rceil$, 
it holds that
  \begin{align*}
1 \le \gamma(\lceil \frac{n}{\log n}\rceil+1) \frac{\log n}{\pi}(1+o(1)) +O(1/\log n).
 \end{align*}
So we get 
  \begin{align}\label{form2}
\gamma(\lceil \frac{n}{\log n}\rceil+1) \ge \frac{\pi}{\log n}(1+o(1)).
 \end{align}
By (\ref{form1}) and (\ref{form2}) we get the result. 
\end{proof}

\begin{proof}[Proof of Theorem \ref{conv3}]
We write $L_n$ as
 \begin{align*}
L_n = \sum_{k=0}^{n} 1_{C_{k,n}}(\omega),
 \end{align*}
where $C_{k,n}$ is the event that 
$\{S_m\}_{m=0}^n \not\supset {\cal N}(S_k)$ and $ S_k\notin\{S_{m}\}_{m=k+1}^n$. 
If we denote by $C'_{k,n,j}$ the event that 
$S_k+b_j\notin\{S_m\}_{m=0}^{k-1}$ and $\{S_m\}_{m=k+1}^{n}\cap \{S_k,S_k+b_j\} = \emptyset$, 
we find that
 \begin{align}\label{j1}
C_{k,n}=\bigcup_{i=1}^4 C'_{k,n,j}.
 \end{align}
So we get
 \begin{align}\label{form3}
P(C'_{k,n,j})\le P(C_{k,n})\le \sum_{j=1}^4P(C'_{k,n,j}).
 \end{align}
It holds that
 \begin{align*}
P(C'_{k,n,j})=&P(S_k+b_j\notin\{S_m\}_{m=0}^{k-1})
\times P(\{S_m\}_{m=k+1}^{n}\cap \{S_k,S_k+b_j\} = \emptyset)\\
=&P(b_j\notin\{S_m\}_{m=1}^{k})
\times P(\{S_m\}_{m=1}^{n-k}\cap \{0,b_j\} = \emptyset)
 \end{align*} 
Note that  by \cite{Dvo1} it holds that 
 \begin{align}\label{th}
P(b_j\notin\{S_m\}_{m=1}^{k})
=P(0\notin\{S_m\}_{m=0}^{k+1})
\sim\frac{\pi}{\log k}. 
 \end{align} 
Therefore, by summing over $k$ in (\ref{form3}) with the help of Lemma \ref{local} and (\ref{th}) we get  (\ref{s1}), provided that the limit in it exists. 
Next we show (\ref{s2}) (in the same sense as for (9)). 
If $l=\{l_i\}_{i=1}^p$  and $D_{l,n}$ is the event that 
$\{S_m\}_{m=0}^{n} \not\supset {\cal N}( S_{l_1})$, $S_{l_1}=S_{l_2}=...=S_{l_p}$ 
and $S_{l_1}\not\in (\{S_m\}_{m=0}^{n}-\{S_{l_i}\}_{i=1}^p)$,
then  it holds that 
 \begin{align*}
J_n^{(p)}=\sum_{l_1=0}^{n} \sum_{l_2=l_1+2}^{n}...\sum_{l_p=l_{p-1}+2}^n 1_{D_{l,n}}(\omega).
 \end{align*}
If $D'_{l,n,j}$ denotes  the event that 
 $$S_{l_1}=S_{l_2}=...=S_{l_p}\quad 
\mbox{and}\quad 
  (\{S_m\}_{m=0}^{n}-\{S_{l_i}\}_{i=1}^p)\cap \{S_{l_1},S_{l_1}+b_j \} = \emptyset,$$ 
then 
 \begin{align}\label{j2}
D_{l,n}=\bigcup_{j=1}^4 D'_{l,n,j}.
 \end{align}
So we get
 \begin{align}\label{form4}
P(D'_{l,n,j})\le P(D_{l,n})\le \sum_{j=1}^4P(D'_{l,n,j}).
 \end{align}
It holds that 
 \begin{align*}
&P(D'_{l,n,j})
\\=&P(\{S_m\}_{m=0}^{l_1-1}\cap \{S_{l_1},S_{l_1}+b_j \} = \emptyset)\\
\times &P(\{S_m\}_{m=l_1+1}\text { firstly hit }S_{l_1} \text{ at the time }l_2
\text{ and }S_{l_1}+b_j\notin\{S_m\}_{m=l_1+1}^{l_2})\\
\times&...\times P(\{S_m\}_{m=l_{p-1}+1}\text { firstly hit }S_{l_{p-1}} \text{ at the time }l_p
\text{ and }S_{l_{p-1}}+b_j\notin\{S_m\}_{m=l_{p-1}+1}^{l_p})\\
\times &P( \{S_m\}_{m=l_p+1}^{n}\cap \{S_{l_p},S_{l_p}+b_j\} = \emptyset)\\
=&P(\{S_m\}_{m=1}^{l_1}\cap \{0,b_j\} = \emptyset)\\
\times &P(\{S_m\}_{m=1}\text { firstly hit }0 \text{ at the time }l_2-l_1
\text{ and }b_j\notin\{S_m\}_{m=1}^{l_2-l_1})\\
\times&...\times P(\{S_m\}_{m=1}\text { firstly hit }0 \text{ at the time }l_p-l_{p-1}
\text{ and }b_j\notin\{S_m\}_{m=1}^{l_p-l_{p-1}} )\\
\times &P(\{S_m\}_{m=1}^{n-l_p} \cap \{0,b_j\} = \emptyset).
 \end{align*}
 We compute the upper bound of (\ref{s2}).  
Summing over $l_2,...,l_p$ in (\ref{form4}) we get
 \begin{align}\notag
&\sum_{l_2=l_1+2}^{l_1+\lceil \frac{n}{\log n} \rceil}...\sum_{l_p=l_{p-1}+2}^{l_{p-1}+\lceil \frac{n}{\log n} \rceil} P(D'_{l,n,j})\\
\notag
\le&P(\{S_m\}_{m=1}^{l_1}\cap \{0,b_j\} = \emptyset)
\times P(0\in\{S_m\}_{m=1}^{\infty},b_j\notin\{S_m\}_{m=1}^{T_0} )\\
\notag
\times&...\times P(0\in\{S_m\}_{m=1}^{\infty},b_j\notin\{S_m\}_{m=1}^{T_0})\\
\label{k1}
\times &P(\{S_m\}_{m=1}^{n-(p-1)\lceil \frac{n}{\log n} \rceil-l_1}\cap \{0,b_j\} = \emptyset).
 \end{align}
It is shown in [3] (see (2.4) of it) that 
 \begin{align*}
&P(\lceil \frac{n}{\log n} \rceil<T_0\le n)\\
=&P(T_0>\frac{n}{\log n})-P(T_0> n)\\
\le &(\frac{\pi}{\log n-\log \log n}+\frac{C\log \log n}{(\log n-\log \log n)^2})-(\frac{\pi}{\log n}-\frac{C\log \log n}{(\log n)^2})
\le \frac{C'\log \log n}{(\log n)^2} ,
 \end{align*}
for some constants $C$ and  $C'$. 
Hence  we can obtain  the bound 
 \begin{align}\notag
&\sum_{l_2=l_1+2}^{l_1+n}...\sum_{l_p=l_{p-1}+2}^{l_{p-1}+n} P(D'_{l,n,j})-
\sum_{l_2=l_1+2}^{l_1+\lceil \frac{n}{\log n} \rceil}...\sum_{l_p=l_{p-1}+2}^{l_{p-1}+\lceil \frac{n}{\log n} \rceil} P(D'_{l,n,j})\\
\notag
=&\sum_{v=2}^p\sum_{l_2=l_1+2}^{n}...\sum_{l_v=l_{v-1}+\lceil \frac{n}{\log n} \rceil+1}^{n}...\sum_{l_p=l_{p-1}+2}^{n} P(D'_{l,n,j})\\
\label{k1+}
\le &(p-1)P(\{S_m\}_{m=1}^{l_1}\cap \{0,b_j\} = \emptyset)
\times P(\lceil \frac{n}{\log n} \rceil<T_0\le n).
 \end{align}
Summing  over $l_1$ in (\ref{k1}) and (\ref{k1+}) with the help of   Lemma \ref{local}  and (\ref{th}) we get the upper bound of (\ref{s2}). 

To compute the lower bound of (\ref{s2}),  
for $\epsilon>0$ pick $s<\infty$ such that $P(T_0<T_{b_j}, T_0<s)>P(T_0<T_{b_j})-\epsilon$. 
It holds that for $n\ge(p-1)s$
  \begin{align*}
&\sum_{l_2=l_1+2}^{n}...\sum_{l_p=l_{p-1}+2}^n P(D'_{l,n,j})\\
\ge&P(\{S_m\}_{m=1}^{l_1}\cap \{0,b_j\} = \emptyset)\\
\times &P(0\in\{S_m\}_{m=1}^{s},b_j\notin\{S_m\}_{m=1}^{T_0})
\times...\times P(0\in\{S_m\}_{m=1}^{s},b_j\notin\{S_m\}_{m=1}^{T_0})\\
\times &P(\{S_m\}_{m=1}^{n-(p-1)s-l_1}\cap \{0,b_j\} = \emptyset)\\
\ge &P(\{S_m\}_{m=1}^{l_1}\cap \{0,b_j\} = \emptyset)
\times (\tilde{c}-\epsilon)^{p-1} 
\times P(\{S_m\}_{m=1}^{n-l_1}\cap \{0,b_j\} = \emptyset) .
 \end{align*}
Therefore, by summing over $l_1$, by Lemma \ref{local} we get the lower bound of (\ref{s2}). 

Also,  if we set $l=\{l_i\}_{i=1}^p$, then 
 \begin{align*}
J_n^p=\sum_{l_1=0}^{n} \sum_{l_2=l_1+2}^{n}...\sum_{l_p=l_{p-1}+2}^n 1_{E_{l,n}}(\omega),
 \end{align*}
where $E_{l,n}$ is the event that 
$\{S_m\}_{m=0}^{n} \not\supset {\cal N} (S_{l_1})$, $S_{l_1}=S_{l_2}=...=S_{l_p}$ 
and $S_{l_1}\not\in (\{S_m\}_{m=l_1+1}^{n}-\{S_{l_i}\}_{i=2}^p)$. 
So we can verify  (\ref{s3}) by the argument given for  (\ref{s2}). 
We finally prove the existence of the limits. 
 \cite{SP2}, Theorem $2$ tells us that for any $a\in {\mathbb{Z}^2}$, 
$\lim_{n\to \infty}P(a\notin \{S_i\}_{i=1}^n)\times (\log n)$ exists. 
Since by applying inclusion-exclusion formula (e.g., \cite{Dur}, Exercise $1.6.9$) 
(\ref{j1}) can be divided, it holds that 
  \begin{align*}
\lim_{n\to \infty}EL_n\times\frac{(\log n)^2}{n} \quad \text{exists.}
 \end{align*}
Also, by the same argument  it is easy to see that
 \begin{align*}
\lim_{n\to \infty}EJ_n^{(p)}\times\frac{(\log n)^2}{n}, \quad
\lim_{n\to \infty}EJ_n^p\times\frac{(\log n)^2}{n} \quad \text{exist.}
 \end{align*}
\end{proof}

(Izumi Okada) 

Department of Mathematics, Tokyo Institute of Technology, 
Tokyo, 152-8550, Japan.

okada.i.aa@m.titech.ac.jp

\begin{thebibliography}{99}
\bibitem{itai}
BENJAMINI,I. and KOZMA,G. and YADIN,A. and YEHUDAYOFF,A.(2010). Entropy of random walk range. Ann. Inst. H. Poincar\'{e} Probab. Statist.Volume 46, Number 4, 1080-1092.
\bibitem{Don1}
DONSKER,M.D. and VARADHAN,S.R.S.(1979). On the number of distinct sites visited by a random walk. Comm. Pure Appl. Math. 32, 721-747.


\bibitem{Dur}
DURRET,R.(2010). Probability: theory and example, Edition 4.Cambridge Series.

\bibitem{Dvo1}
DVORETZKY,A. and ERD\H{O}S,P.(1951). Some problems on random walk in space. Proc. Second Berkeley Symp. Math. Statist. Probab. 353-367. Univ. California Press, Berkeley.

\bibitem{Fla}
FLATTO,L.(1976). The Multiple range of two-dimensional recurrent walk.Ann. Probab.
Volume 4, Number 2, 155-338.

\bibitem{Hama2}
HAMANA,Y. and KESTEN,H.(2001). A large-deviation result for the range of random walk and for the Wiener sausage. Probab. Th. Rel. F., June, Volume 120, Issue 2, 183-208.

\bibitem{Hama3}
HAMANA,Y.(2001). Asymptotics of the moment generating function for the range of random walks. J. Theoret. Probab.January, Volume 14, Issue 1, 189-197.

\bibitem{Hama4}
HAMANA,Y.(1997). The fluctuarion result for the multiple point range of two dimensional recurrent random walk, Ann. Probab. 25, 598-639.
\bibitem{Hama5}
HAMANA, Y.(1998). A remark on the multiple point range of two dimensional random walks, Kyusyu J. Math. 52, 23-80.  
\bibitem{Jain1}
JAIN,N.C. and PRUITT,W.E.(1970). The range of recurrent random walk in the plane. Z. Wahrsch. Verw. Gebiete 16 279-292.

\bibitem{Jain2}
JAIN,N.C. and PRUITT,W.E.(1971). The range of transient random walk. J.Anal.Math.24. 369-393.

\bibitem{Jain3}
JAIN,N.C. and PRUITT,W.E.(1972). The range of random walk. Proc.Sixth Berkeley Symp. Math. Statist. Probab. 3. 31-50. Univ. California Press, Berkeley.


\bibitem{Jain4}
JAIN,N.C. and PRUITT,W.E.(1974). Further limit theorems for the range of random walk. J.Anal.Math.27. 94-117.



\bibitem{SP2}
KESTEN,H. and SPITZER,F.(1963).
Ratio theorems for random walks I.Journal d' Analyse Math\'{e}matique
December, Volume 11, Issue 1, 285-322.

\bibitem{gall}
LE GALL, J.-F.(1986). Propri\'{e}t\'{e}s d'intersection des marches al\'{e}atoires I, Comm, Math. Phys. 104, 451-507. 

\bibitem{Law}
LAWLER,G.F.(1991). Intersections of Random Walks. Birkhauser, Boston.

\bibitem{pitt}
PITT, J.H.(1974). Multiple points of transient random walk, Proc. Amer. Math. Soc. 43, 195-199. 

\bibitem{SP}
SPITZER,F.(1976). Principles of Random Walk.Springer,Berlin.

\end{thebibliography}
\end{document}